\documentclass[12pt,centertags,oneside]{amsart}
\usepackage{amsmath,amstext,amsthm,amscd,typearea}
\usepackage{amssymb}
\usepackage{fancyhdr}
\usepackage[mathscr]{eucal}
\usepackage{mathrsfs}
\usepackage{concmath}
\usepackage{charter}
\usepackage{dsfont} 
\usepackage{hyperref}
\usepackage{color}
\usepackage{titletoc}

\usepackage[a4paper,width=16.2cm,top=3cm,bottom=3cm]{geometry}

\numberwithin{equation}{section}

\newcommand{\field}[1]{\mathbb{#1}}

\newcommand{\C}{\field{C}}
\newcommand{\N}{\field{N}}

\def\Im{{\rm Im}}

\newcommand{\cali}[1]{\mathscr{#1}}
\newcommand{\cH}{\cali{H}}
 
\DeclareMathOperator{\Ker}{Ker}

\DeclareMathOperator{\Dom}{Dom}

\newcommand{\ddbar}{\overline\partial}
\newcommand{\dbar}{\partial}
\newtheorem{thm}{Theorem}[section]
\newtheorem{lemma}[thm]{Lemma}

\newtheorem{cor}[thm]{Corollary}
\theoremstyle{definition}
\newtheorem{rem}[thm]{Remark}
\theoremstyle{definition}
\newtheorem{defn}[thm]{Definition}

\newcommand{\be}{\begin{eqnarray}}
\newcommand{\ee}{\end{eqnarray}}

\newcommand{\comment}[1]{}
\begin{document}
\title
{On the growth of von Neumann dimension of harmonic spaces of 
semipositive line bundles over covering manifolds}
\author{Huan Wang}
\address{Universit{\"a}t zu K{\"o}ln,  Mathematisches Institut,
    Weyertal 86-90,   50931 K{\"o}ln, Germany}

\date{16.05.2015}

\begin{abstract}
We study the harmonic space of line bundle valued forms over a covering manifold with a discrete group action $\Gamma$, and obtain an asymptotic estimate for the $\Gamma$-dimension of the harmonic space with respect to the tensor times $k$ in the holomorphic line bundle $L^{k}\otimes E$ and the type $(n,q)$ of the differential form, when $L$ is semipositive. In particular, we estimate the $\Gamma$-dimension of  the corresponding reduced $L^2$-Dolbeault cohomology group. Essentially, we obtain a local estimate of the pointwise norm of harmonic forms with valued in semipositive line bundles over Hermitian manifolds.
\end{abstract}
 
\maketitle 
         \section{Introduction}\label{intro}
\ 
The purpose of this paper is to study the growth of the von Neumann dimension of the space
of harmonic forms with values in powers of an invariant semipositive line bundle over
a Galois covering of a compact Hermitian manifold.
The main technical tool will be an estimate of the Bergman kernel on a compact
set of a Hermitian manifold, which generalizes a result of Berndtsson \cite{BB:02}
for compact manifolds. 
  
Let $(X,\omega)$ be a Hermitian (paracompact) manifold of dimension $n$ and
$(L,h^L)$ and $(E,h^E)$ be Hermitian holomorphic line bundles over $X$.
For $k\in\N$ we form the Hermitian line bundles $L^k:=L^{\otimes k}$ and
$L^{k}\otimes E$, the latter
endowed with the metric $h_k=(h^L)^{\otimes k}\otimes h^E$.

To the metrics $\omega$, $h^L$ and $h^E$ we associate the Kodaira Laplace operator $\square_k$
acting on forms with values in $L^{k}\otimes E$ and also
$L^2$ spaces of forms with values in $L^{k}\otimes E$, and 
$\square_k$  
has a (Gaffney) self-adjoint extension in the space of $L^2$-forms, 
denoted by the same symbol.

The space $\cH^{p,q}(X,L^{k}\otimes E)$ of harmonic $L^{k}\otimes E$-valued 
$(p,q)$-forms is defined as the kernel of (the self-adjoint extension of) $\square_k$ 
acting on the $L^2$ space of  $(p,q)$-forms.
 
In this paper we mainly work with $(n,q)$-forms.
Since $\cH^{n,q}(X,L^{k}\otimes E)$ is separable,
let $\{s^k_j\}_{j\geq1}$ be an orthonormal basis and denote by $B_k^q$
the Bergman density function defined by
\begin{equation}\label{e:Bergfcn}
B_k^q(x)=\sum_{j=1}^\infty|s^k_j(x)|_{h_k,\omega}^2\,,
\;x\in X,
\end{equation}
where $|\cdot|_{h_k,\omega}$ is the pointwise norm of a form. 
Definition \eqref{e:Bergfcn} is independent of the choice of basis.

The first main result of this paper is a uniform estimate of the Bergman density function
for semipositive line bundles in a neighborhood of a compact subset of a Hermitian manifold.

\begin{thm}\label{thm1}
Let $(X,\omega)$ be a Hermitian manifold of dimension $n$ and
$(L,h^L)$ and $(E,h^E)$ be Hermitian holomorphic line bundles over $X$. 
Let $K\subset X$ be a compact subset and assume that $(L,h^L)$ is semipositive on a neighborhood
of $K$.

Then there exists  $C>0$ depending on the compact set $K$, the metric $\omega$ 
and the bundles $(L,h^L)$ and $(E,h^E)$, such that for any $x\in K$, $k\geq 1$ and $q\geq 1$,
\begin{equation}\label{e:i2}
B^q_k(x) \leq C k^{n-q}\,,
\end{equation}
where $B^q_k(x)$ is the Bergman kernel function \eqref{e:Bergfcn} of harmonic 
$(n,q)$-forms with values in $L^k\otimes E$. 
\end{thm}

For $X$ compact and $K=X$, Theorem \ref{thm1} reduces to \cite[Theorem {2.3}]{BB:02}. 
Theorem \ref{thm1} will be used to obtain the following
bounds for the von Neumann dimension of the harmonic spaces
on covering manifolds. 

\begin{thm}\label{thm2}
Let $(X,\omega)$ be a Hermitian manifold of dimension $n$ on which a discrete 
group $\Gamma$ acts holomorphically, freely and properly such that $\omega$ is a $\Gamma$-invariant Hermitian 
metric and the quotient $X/\Gamma$ is compact. Let $(L,h^L)$ and $(E,h^E)$ be two $\Gamma$-invariant 
holomorphic Hermitian line bundles on $X$. Assume $(L,h^L)$ is semipositive on $X$.
Then there exists $C>0$ such that for any 
$q\geq 1$ and  $k\geq 1$ we have
\begin{equation}\label{eq2}
\dim_{\Gamma}\cH^{n,q}(X,L^k\otimes E)\leq C k^{n-q}\,,\:\:
\dim_{\Gamma}\cH^{0,q}(X, L^k\otimes E)
\leq C k^{n-q}\,.
\end{equation}
The same estimate also holds for the reduced $L^2$-Dolbeault cohomology groups,
\begin{equation}\label{eq16}
\dim_{\Gamma}{\overline{H}}^{0,q}_{(2)}(X, L^k\otimes E)
\leq C k^{n-q}.
\end{equation}
\end{thm}
For a compact manifold $X$, the growth of the dimension of the Dolbeault cohomology 
$H^{0,q}(X,L^k\otimes E)$ as $k\to\infty$ is of fundamental importance in algebraic
and complex geometry and is linked to the structure of the manifold, cf.\ \cite{Dem,Dem:85,MM}.
If $(L,h^L)$ is positive, then $H^{0,q}(X,L^k\otimes E)=0$ for $q\geq1$ and $k$ large enough,
by the Kodaira-Serre vanishing theorem \cite[Theorem\,1.5.6]{MM}. This reflects the
fact that the remaining cohomology space $H^{0,0}(X,L^k\otimes E)$ is rich enough
to provide a projective embedding of $X$, for large $k$.

Assume now that $(L,h^L)$ is semipositive. The solution of the Grauert-Riemenschneider conjecture
by Demailly \cite{Dem:85} and Siu \cite{Sil:84} shows that $\dim H^{0,q}(X,L^k\otimes E)=o(k^n)$ 
as $k\to\infty$ for $q\geq1$. This can be used to show that $X$ is a Moishezon manifold,
if $(L,h^L)$ is moreover positive at at least one point.
Berndtsson \cite{BB:02} showed that we have actually $\dim H^{0,q}(X,L^k\otimes E)=O(k^{n-q})$ 
as $k\to\infty$ for $q\geq1$. Note that the latter estimate can be
proved by induction on the dimension if $X$ is projective, see \cite[(6.7) Lemma]{Dem:94}. 
For the Bergman kernel $B^0_k$ on $(n,0)$-forms with values in a semipositive line bundle,
it was shown by Hsiao-Marinescu \cite[Theorem 1.7]{HsM} that it has an asymptotic expansion
on the set where the curvature is strictly positive.
 
The study of $L^2$ cohomology spaces on coverings of compact manifolds
has also interesting applications, cf.\ \cite{GHS:98,Ko:95}.
The results are similar to the case of compact manifolds, but
we have to use the reduced $L^2$ cohomology groups and von Neumann dimension 
instead of the usual dimension.
For example, in the situation of Theorem \ref{thm2}, if the invariant line bundle $(L,h^L)$
is positive, the Andreotti-Vesentini vanishing theorem \cite{AV:65} 
shows that ${\overline{H}}^{0,q}_{(2)}(X, L^k\otimes E)\cong\cH^{0,q}(X, L^k\otimes E)=0$ 
for $q\geq1$ and $k$ large enough.
The holomorphic Morse inequalities of Demailly \cite{Dem:85} were generalized to coverings
by Chiose-Marinescu-Todor \cite{MTC:02,TCM:01} (cf.\ also \cite[(3.6.24)]{MM}) and yield in the conditions
of Theorem \ref{thm2} that $\dim_{\Gamma} H^{0,q}_{(2)}(X,L^k\otimes E)=o(k^n)$ 
as $k\to\infty$ for $q\geq1$. Hence Theorem \ref{thm2} generalizes \cite{BB:02}
to covering manifolds and refines the estimates obtained in \cite{MTC:02,TCM:01}.

Note also that the magnitude $k^{n-q}$ in \eqref{eq2} cannot be improved in general 
\cite[Proposition 4.2]{BB:02}.

Our paper is organized in the following way. In the section \ref{pre}, we introduce the notations and recall the necessary facts. In the section \ref{pro}, we prove some properties of harmonic line bundle valued forms, including $\dbar\ddbar$-formulas on non-compact manifolds and submeanvalue formulas, which imply Theorem \ref{thm1}. In the section \ref{proof}, we prove our main results and a corollary, and explain that Theorem \ref{thm1} implies Theorem \ref{thm2}.

\section{Preliminaries}\label{pre}
  
We introduce here the notations and recall the necessary
facts used in the paper. Let $(X, \omega)$ be a Hermitian manifold of dimension $n$ and
$(L,h^L)$ and $(F, h^F)$ be Hermitian holomorphic line bundles over $X$. Let $\Omega^{p,q}(X, F)$ be the space of smooth $(p,q)$-forms on $X$ with values in $F$ for $p,q\in \N$. The curvature of $(F, h^F)$ is defined by $R^F=\ddbar\dbar \log|s|^2_{h^{L}}$ for any local holomorphic frame $s$, then the Chern-Weil form of the first Chern character of $F$ is $c_1(L, h^L)=\frac{\sqrt{-1}}{2\pi}R^F$, which is a real $(1,1)$-forms on $X$. The volume form is given by $dv_{X}=\omega_n $ where $\omega_q:=\frac{\omega^q}{q!}$ for $1\leq q\leq n$.

We will use several times the notion of positive $(p,p)$-form, for which we refer to \cite[Chapter III, \S 1, (1.1) (1.2)(1.5)(1.7)]{Dem}. If a $(p,p)$-form $T$ is positive, we write $T\geq 0$. Given two $(p,p)$-forms $T_1$, $T_2$ we say that $T_1\geq T_2$, if $T_1-T_2\geq 0$. From definitions the following statement follows easily: if $T\geq 0$ is a $(p,p)$-form and $u\geq 0$ is a $(1,1)$-form, then $T\wedge u\geq 0$.
 
 The $L^2$-scalar product on $\Omega^{p,q}(X, F)$ is given by $\langle s_1,s_2 \rangle=\int_X \langle s_1(x), s_2(x) \rangle_{h^F,\omega} dv_X(x) $, where $\langle ,\rangle_{h^F,\omega}$ is the pointwise Hermitian inner product induced by $\omega$ and $h^F$. We denote by $L^2_{p,q}(X, F)$, the $L^2$ completion of $\Omega^{p,q}_0(X, F)$, which is the subspace of $\Omega^{p,q}(X, F)$ consisting of elements with compact support.
  
Let $\ddbar^{F}: \Omega_0^{p,q} (X, F)\rightarrow L^2_{p,q+1}(X,F) $ be the Dolbeault operator and let  $ \ddbar^{F}_{\max} $ be its maximal extension (see \cite[Lemma 3.1.1]{MM}). From now on we still denote the maximal extension by $ \ddbar^{F} :=\ddbar^{F}_{\max} $ and  the associated Hilbert space adjoint by $\ddbar^{F*}:=(\ddbar^{F}_{\max})_H^*$ for simplifying the notations. Consider the complex of closed, densely defined operators
 $L^2_{p,q-1}(X,F)\xrightarrow{\ddbar^{F}}L^2_{p,q}(X,F)\xrightarrow{\ddbar^{F}} L^2_{p,q+1}(X,F)$,
then $(\ddbar^{F})^2=0$.

By \cite[Proposition.3.1.2]{MM}), the operator defined by
\begin{eqnarray}\label{eq40} \nonumber
	\Dom(\square^{F})&=&\{s\in \Dom(\ddbar^{F})\cap \Dom(\ddbar^{F*}): \ddbar^{F}s\in \Dom(\ddbar^{F*}),~\ddbar^{F*}s\in \Dom(\ddbar^{F}) \}, \\ 
	\square^{F}s&=&\ddbar^{F*} \ddbar^{F}s+\ddbar^{F*} \ddbar^{F}s \quad \mbox{for}~s\in \Dom(\square^{F}),
\end{eqnarray}
is a positive, self-adjoint extension of Kodaira Laplacian, called the Gaffney extension.

\begin{defn} 
The space of harmonic forms $\cH^{p,q}(X,F)$ is defined by 
\begin{equation}\label{eq45}
\cH^{p,q}(X,F):=\Ker(\square^{F})=\{s\in \Dom(\square^{F})\cap L^2_{p,q}(X, F): \square^{F}s=0 \}.
\end{equation}
The $q$-th reduced $L^2$-Dolbeault cohomology is defined by 
\begin{equation}\label{eq46}
\overline{H}^{0,q}_{(2)}(X,F):=\dfrac{\Ker(\ddbar^{F})\cap  L^2_{0,q}(X,F) }{[ \Im( \ddbar^{F}) \cap L^2_{0,q}(X,F)]},
\end{equation}
where $[V]$ denotes the closure of the space $V$.
\end{defn} 
According to the general regularity theorem of elliptic operators (also cf.\cite[Theorem A.3.4]{MM}), 
$s\in \cH^{p,q}(X,F) $ implies $s\in\Omega^{p,q}(X,F)$. The Bergman kernel function defined in \eqref{e:Bergfcn} 
is well-defined by an adaptation of 
\cite[Lemma 3.1]{CM} to form case. By weak Hodge decomposition (cf.\ \cite[(3.1.21)(3.1.22)]{MM}), 
we have a canonical isomorphism  
\begin{equation}\label{eq47}
\overline{H}^{0,q}_{(2)}(X,F)\cong \cH^{0,q}(X,F)
\end{equation} for any $q\in \N$, which associates to each cohomology class its unique harmonic representative. 
 
A Hermitian manifold $(X,\omega)$ is called complete, if all geodesics are defined for all time for the underlying Riemannian manifold. By \cite[Corollary 3.3.3, 3.3.4]{MM}, if $X$ is complete, then $\square^F|_{\Omega^{p,q}(X,F)}$ is essentially self-adjoint.

           \section{Some properties of harmonic line bundle valued forms}\label{pro}
           
In this section, we work under the following general setting, and later the covering manifold with a group action $\Gamma$ will be treated as a special case in the section \ref{proof}. Let $(X ,\omega)$ be a Hermitian manifold of dimension $n$ and $(F, h^{F})$ be a holomorphic Hermitian line bundle on $X$. For the Kodaira Laplacian $\square:=\square^{F}$ we denote still by $\square$ its (Gaffney) self-adjoint extension. 

          \subsection{The $\dbar\ddbar$-Bochner formula for non-compact manifolds} 
          \  
The following $\dbar\ddbar$-formula was obtained by B. Berndtsson in \cite{BB:02} 
and \cite{BB:03} for $\ddbar$-closed, line bundle valued, $(n,q)$- forms over compact manifolds. 
We can rephrase it for $\square$-closed, line bundle valued, $(n,q)$- forms over any compact subset 
of a Hermitian (possibly non-compact) manifold. 
The proof is analogue to \cite[Proposition 2.2]{BB:02} and \cite[Proposition 6.2]{BB:03}, 
thus we omit it here. 
 
Let $\alpha\in\cH^{n,q}(X,F)$. We define now a positive
$(n-q,n-q)$ form $T_\alpha$ on $X$ as follows.
Let $U$ be an open set such that $F|_U$ is trivial and let
$e_F$ be a local holomorphic frame on $U$ and set $|e_F|^2_{h^{F}}=e^{-\psi}$.
Write $\alpha|_U=\xi\otimes e_F$, with $\xi\in\Omega^{n,q}(U,\C)$. 
The $(n-q,n-q)$ form $T_\alpha$ is defined locally by
$T_\alpha|_U:= i^{{(n-q)}^2}(\star\xi)\wedge\overline{(\star\xi)}e^{-\psi}$, where $\star:\Omega^{n,q}(U,\C) \to \Omega^{n-q, 0}(U,\C) $
is the Hodge star operator associated to the metric $\omega$ given by $\xi\wedge\overline{\star\xi}=|\xi|^2_{\omega}\omega_n$. 
It is easy to check that $T_{\alpha}$ is well defined globally.

We have a $F$-valued $(n-q,0)$ form $\gamma_{\alpha} \in \Omega^{n-q, 0}(X,F)$ 
associated to $\alpha$ defined locally by ${\gamma_{\alpha}}|_U:=(\star\xi)\otimes e_F$, 
which is also well defined globally. Let $L:=\omega \wedge \cdot$ be the Lefschetz operator 
on $\Omega^{p,q}(X,F)$ and let $\Lambda$ be its dual operator defined by 
$ \langle \Lambda\cdot,\cdot \rangle_{h^F,{\omega}}=\langle \cdot,L\cdot \rangle_{h^F,{\omega}}$.
  
\begin{thm}\label{lem1}
Let $(F,h^{F})$ be a holomorphic Hermitian line bundle over a Hermitian manifold 
$(X,\omega)$. Assume $\alpha\in\cH^{n,q}(X,F)$, $q\geq 1$, and $K\subset X$ 
is a compact subset. Then there exist non-negative constants $C_1$ and $C_2$ depending on $\omega$ and $K$, such that 
\begin{eqnarray}\label{eq3} 
i\dbar\ddbar(T_{\alpha}\wedge \omega_{q-1})
&\geq&
 (\langle 2\pi c_1(F, h^F)\wedge \Lambda \alpha, \alpha \rangle_{h^F,\omega} 
 -C_1|\alpha|_{h^F,\omega}^{2}+C_2|\ddbar^{F}{\gamma_{\alpha}}|_{h^F,\omega}^2)\omega_n \\ \nonumber
 &\geq& (\langle 2\pi c_1(F, h^F)\wedge \Lambda \alpha, \alpha \rangle_{h^F,\omega}-C_1|\alpha|_{h^F,\omega}^{2})\omega_n
\end{eqnarray}
on $K$. Here $\langle c_1(F, h^F) \wedge\Lambda\alpha,\alpha \rangle_{h^F,\omega} \omega_n
= c_1(F, h^F)\wedge T_{\alpha}\wedge\omega_{q-1}$ on $X$. In particular, if $X$ is K\"{a}hler, then 
\begin{eqnarray}\label{eq112}
i\dbar\ddbar(T_{\alpha}\wedge \omega_{q-1})
&=&  
(i\dbar\ddbar T_{\alpha})\wedge \omega_{q-1} \\ \nonumber
&=& (\langle 2\pi c_1(F, h^F)\wedge \Lambda \alpha, \alpha \rangle_{h^F,\omega} +|\ddbar^F{\gamma_{\alpha}}|_{h^F,\omega}^2)\omega_n 
\end{eqnarray}
on $K$. Above $\langle\,\cdot\,,\cdot\,\rangle_{h^F,\omega}$ and $|\cdot|_{h^F,\omega}^2$ denote the pointwise Hermitian metric and norm on $F$-valued differential forms induced by $\omega$ and $h^F$. 
\end{thm}  
 
Based on the same argument as in Theorem \ref{lem1} and the $\ddbar$-closed case in 
\cite[Proposition 2.2]{BB:02}, we have the following equality for K\"{a}hler manifolds, 
which generalizes both the above formula (\ref{eq112}) and the K\"{a}hler case of \cite[Proposition 2.2]{BB:02}. 
For compact K\"{a}hler manifolds, a general formula of this type can be found in \cite{BB:10}.
\begin{cor}\label{cor1} 
Let $(F, h^F)$ be a holomorphic Hermitian
line bundle over a K\"{a}hler manifold $(X,\omega)$. Assume $\alpha \in \Omega^{n,q}(X,F)\cap \Dom(\ddbar^F)\cap\Dom(\ddbar^{F*})$ such that $\ddbar^F \alpha=0$ and $q\geq 1$. Then,
\begin{eqnarray}\label{eq163}
i\dbar\ddbar(T_{\alpha}\wedge\omega_{q-1})
&=&(i\dbar\ddbar T_{\alpha}) \wedge\omega_{q-1} \\ \nonumber
&=& -2 Re\langle \ddbar^F \ddbar^{F*}\alpha,\alpha \rangle_{h^F,{\omega}} \omega_n
+ \langle 2\pi c_1(F, h^F) \wedge\Lambda\alpha,\alpha \rangle_{h^F,{\omega}} \omega_n \\ \nonumber
&&+|\ddbar^{F*} \alpha|_{h^F,{\omega}}^2\omega_n
+|\ddbar^F {\gamma_{\alpha}}|_{h^F,{\omega}}^2\omega_n
\end{eqnarray}
where $\langle\,\cdot\,,\cdot\,\rangle_{h^F,\omega}$ and $|\cdot|_{h^F,\omega}^2$ are the pointwise Hermitian metric and norm on $F$-valued differential forms induced by $\omega$ and $h^F$.
\end{cor}   

           \subsection{Submeanvalue formulas of harmonic forms in $\cH^{n,q}(X,L^k\otimes E)$}\label{submeanvalue}
Let $(L,h^L)$ and $(E,h^E)$ be Hermitian holomorphic line bundles over $X$. For any compact subset $K$ in $X$, the interior of $K$ is denoted by $\mathring{K}$. Let $K_1, K_2$ be compact subsets in $X$, such that $K_1\subset \mathring{K_2}$. Then there exists a constant $c_0=c_0(\omega, K_1, K_2)>0$ such that for any $x_0\in K_1$, the local exponential normal coordinate around $x_0$ is $V\cong W\subset \C^n$, where $$W:=B(c_0):=\{ z\in\C^n: |z|<c_0\}, \quad V:=B(x_0,c_0)\subset \mathring{K_2}\subset K_2,$$ $z(x_0)=0$, and
 $\omega(z)= \sqrt{-1} \sum_{i,j} h_{ij}(z)d{z_i} \wedge d{\overline{z}_j}$ with $h_{ij}(0)=\frac{1}{2}\delta_{ij}$ and $d h_{ij}(0)=0$. 

\begin{lemma}\label{lem2}
Let $(X,\omega)$ be a Hermitian manifold of dimension $n$ and
$(L,h^L)$ and $(E,h^E)$ be Hermitian holomorphic line bundles over $X$. 
Let $K_1$ and $K_2$ be compact subsets in $X$ such that $K_1\subset \mathring{K_2}$. Assume  $L \geq 0$ in $\mathring{K_2}$  and $q\geq 1$. Then there exists a constant $C>0$ depending on $\omega$, $K_1$, $K_2$ and $(E ,h_{E})$, such that
\begin{equation}\label{eq4} 
\int_{|z|<r}|\alpha|_{h_k,\omega}^2 dv_X \leq Cr^{2q}\int_{X}|\alpha|_{h_k,\omega}^2 dv_X
\end{equation}
for any $\alpha\in \cH^{n,q}(X,L^k\otimes E)$ and $0<r<\frac{c_0}{2^n}$, where $|\cdot|_{h_{k},\omega}^2$ is the pointwise Hermitian norm induced by $\omega$, $h^L$ and $h^E$.
\end{lemma}
 
\begin{proof}
For simplifying notations, we denote by $\langle \cdot, \cdot \rangle_h$ and $|\cdot|_h$ 
the associated pointwise Hermitian metrics and norms here, and their meaning will be clear from the context. 
For $ 0<t<c_0 $, define $\sigma (t):=\int _{|z|<t}|\alpha|_h^{2}\omega_n = \int_{|z|<t} T_{\alpha}\wedge\omega_q$.
Assume $\|\alpha\|^2_{L^2}:=\int_{X}|\alpha|_h^2 \omega_n=1$. Then this lemma says that: 
There exists a constant $C$, which is independent of the point $x_0$ and $k$ in $L^k\otimes E$, such that
\begin{equation}\label{eq11}
\sigma(r)\leq C r^{2q},
\end{equation}
when $0<r<c_0/2^n$ (eventually we will use the special case $r=\frac{2}{\sqrt{k}}$ as $k\rightarrow \infty$).

From the Theorem \ref{lem1} for $F=L^k\otimes E$, there exist $C_3=C_3(\omega, K_2, E, h_{E})\geq 0$ such that
\begin{eqnarray*}
\langle c_1(F, h^F)\wedge \Lambda \alpha, \alpha \rangle_h \omega_n
&=&c_1(F, h^F)\wedge T_{\alpha}\wedge \omega_{q-1}
=(kc_1(L, h^L)+c_1(E, h^E))\wedge T_{\alpha}\wedge\omega_{q-1} \\
&\geq& c_1(E, h^E)\wedge T_{\alpha}\wedge \omega_{q-1}
=\langle c_1(E, h^E)\wedge \Lambda \alpha, \alpha \rangle_h \omega_n  \\
&\geq& -C_3|\alpha|_h^{2}\omega_{n}
\end{eqnarray*}
on $\mathring{K_2}$, since $L\geq 0$, $T_{\alpha}\geq 0$ and $\omega$ is positve Hermitian (1,1)-form on $\mathring{K_2}$. Thus over $\mathring{K_2}$, (\ref{eq3}) becomes
\begin{equation}\label{eq5}
i\dbar\ddbar(T_{\alpha}\wedge \omega_{q-1})
\geq -C_4|\alpha|_h^{2}\omega_n
\end{equation} 
 where $C_4=C_4(\omega,K_2,E,h_{E})\geq 0$. Then 
\begin{eqnarray}\label{eq6}
\int_{|z|<t}(t^2-|z|^2)i\dbar\ddbar(T_{\alpha} \wedge \omega_{q-1})
&\geq&-C_4t^2\sigma(t).
\end{eqnarray} 
Denote $\beta:=\frac{\sqrt{-1}}{2}\dbar\ddbar|z|^2=\frac{\sqrt{-1}}{2}\sum_j dz_j\wedge d\overline{z}_j$, and apply Stokes' formula to the left hand side of (\ref{eq6}), that is
\begin{eqnarray*}
\int_{|z|<t}(t^2-|z|^2)i\dbar\ddbar(T_{\alpha} \wedge \omega_{q-1})
=\int_{|z|<t}\dbar|z|^2\wedge i\ddbar(T_{\alpha} \wedge \omega_{q-1}).
\end{eqnarray*}
Thus 
\begin{eqnarray}\label{eq7}
 2\int_{|z|<t}T_{\alpha}\wedge\omega_{q-1}\wedge \beta
\leq \int_{|z|=t}-i T_{\alpha}\wedge\omega_{q-1}\wedge \dbar|z|^2 
 +C_4 t^2\sigma(t). 
\end{eqnarray}
By the choice of exponential normal coordinates, 
\begin{equation}\label{eq147}
h_{ij}(z)=\frac{1}{2}\delta_{ij}+O(|z|^2).
\end{equation}
for any $z\in B(c_0)$.
In particular, for $|z|=t$ with $0\leq t <c_0$, we can approximate the metric $\omega$ on $X$ by the standard one $\beta$ on $\C^n$ in the following sense
\begin{eqnarray}\label{eq148}
(1-R_1(t))\beta\leq\omega(z)\leq(1+R_1(t))\beta 
\end{eqnarray}
by the smoothness of $\omega$, where $R_1(t)\geq 0$ and $R_1(t)=O(t^2)$ as $t\rightarrow 0$. (Trivially if $X=\C^n$ then $R_1(t)=0$)
Hence
\begin{eqnarray}\label{eq8}
 T_{\alpha}\wedge\omega_{q-1}\wedge\beta
=q (1-R_1(t)) |\alpha|_h^2 \omega_n.
\end{eqnarray} 
Notice $(1+R_1(t))^{-n}\leq\frac{\omega^n}{\beta^n}\leq (1+R_1(t))^n$ and 
\begin{equation}\label{eq172}
\int_{|z|=t}-i T_{\alpha}\wedge\beta_{q-1}\wedge \dbar|z|^2 \leq t\int_{|z|=t}|\alpha|_h^2 dS
\end{equation}
where $dS$ is the surface measure. Then (\ref{eq148}), (\ref{eq8}) and (\ref{eq172}) imply
\begin{eqnarray}\label{eq9}
\int_{|z|=t}-i T_{\alpha}\wedge\omega_{q-1}\wedge \dbar|z|^2 
\leq t(1+R_2(t)) \sigma'(t) 
\end{eqnarray}
where $\sigma'(t)=\int_{|z|=t}|\alpha|_h^2(\omega_n/\beta_n)dS$ by the definition of $\sigma(t)$, $R_2(t)\geq 0$ and $R_2(t)=O(t^2)$.
Combining (\ref{eq7}),(\ref{eq8}) and (\ref{eq9}), we have 
\begin{eqnarray*}
2q (1-R_1(t))\sigma(t)
\leq t(1+R_2(t)) \sigma'(t)+C_4 t^2\sigma(t), 
\end{eqnarray*}
that is,  $2q (1-R_3(t))\sigma(t)\leq t(1+R_2(t))\sigma'(t)$ for any $0\leq t<c_0$, where $R_3(t)\geq 0$ and $R_3(t)=O(t^2)$. 
Substituting $s(t)^2:=\sigma(t)\geq 0$ and dividing by $2ts(t)$, we obtain
\begin{equation}\label{eq10}
q(\frac{1}{t}-R_4(t))s(t)\leq(1+R_2(t)) s'(t)
\end{equation} 
for $q\geq 1$ and any $0\leq t<c_0$, where $R_4(t)\geq 0$ and $R_4(t)=O(t)$.

Now we only need to prove: There exists $C\geq 0$, such that for any $1\leq q\leq n$ and $0\leq t< \frac{c_0}{2^n}$,
\begin{equation}\label{eq12}
s(t)\leq Ct^q,
\end{equation}
which is equivalent to (\ref{eq11}). Next we fix $q\geq 1$, so we only need to prove that $s(t)\leq Ct^m$ for any $0\leq m\leq q$ and $0\leq t\leq c_0/2^q$ by induction over $m$.
Firstly, for $m=0$, $s^2(t):=\sigma(t):=\int_{|z|<t}|\alpha|_h^2\omega_n\leq 1$ for $0\leq t\leq c_0$. Secondly, assume there exists a constant $C_5>0$ such that $s(t)\leq C_5t^m$ for $0\leq m<q$ and $0\leq t\leq c_0/2^m$. Thirdly, in particular, we consider $1\leq m+1=q$ for $0\leq t < c_0/2^{m+1}$, and thus (\ref{eq10}) becomes 
\begin{equation}\label{eq13}
s'(t)-(m+1)\frac{s(t)}{t}\geq -R_2(t)s'(t)-(m+1)R_4(t)s(t).
\end{equation}   
By the second step of the induction, $s(t)=O(t^m)$ with $s(t)\geq 0$ and $s'(t)=O(t^{m-1})$ with $s'(t)\geq 0$.
And according to (\ref{eq13}), for $0\leq t < c_0/2^{m+1}$,
\begin{eqnarray}\label{eq164}
\left(\frac{s(t)}{t^{m+1}}\right)'=\frac{1}{t^{m+1}}\left[s'(t)-\frac{m+1}{t}s(t)\right]
\geq -R_5(t),
\end{eqnarray}
where $R_5(t)\geq 0$ and $R_5(t)=O(1)$. Integral (\ref{eq164}) from $r$ to $c_0/2^{m+1}$, then
\begin{eqnarray}\label{eq165}
\frac{s(\frac{c_0}{2^{m+1}})}{{(\frac{c_0}{2^{m+1}}})^{m+1}}-\frac{s(r)}{r^{m+1}}
=\int_r^{\frac{c_0}{2^{m+1}}}(\frac{s(t)}{t^{m+1}})'dt
\geq -C_6(m+1,\omega,K_1,K_2,E,h_{E}).
\end{eqnarray}
Finally, for the fixed $q\geq 1$ and any $0\leq r \leq c_0/2^{q}$, we have
\begin{eqnarray}\label{eq166}
\frac{s(r)}{r^q}&\leq& C_5\frac{2^{q}}{c_0}+C_6(q,\omega,K_1,K_2,E,h_{E}).
\end{eqnarray}
Let $q$ run over $\{1,...,n\}$ in (\ref{eq166}). Then there exists $C=C(\omega,K_1,K_2,E,h_{E})\geq 0$ such that
(\ref{eq12}) verifies and also (\ref{eq4}) and (\ref{eq11}).
\end{proof} 
   
We will consider the following trivialization of holomorphic Hermitian line bundles in local charts. 
For any $x_0\in K_1 \subset \mathring{K_2}$, we  fix the exponential normal coordinates on 
$V\cong W\subset \C^n$ as before such that 
$$\omega(x_0)=\beta:=\frac{\sqrt{-1}}{2}\sum d{z_j} \wedge d{\overline{z}_j},$$ 
which is the standard metric on $\C^n$. Let $L\geq 0$ on $\mathring{K_2}$. 
Then we can choose the trivialization of $L$ and $E$ over $V$ such that for any $z\in B(c_0)$, 
$|e_{L}(z)|_{h^L}^2=e^{-\phi(z)}$ and $|e_{E}(z)|_{h^E}^2=e^{-\varphi(z)}$ satisfying 
\begin{equation}\label{eq167}
\phi(z)=\sum \lambda_{i}|z_{i}|^2+O(|z|^{3}), \quad \varphi(z)=\sum \mu_{i}|z_{i}|^2+O(|z|^{3})
\end{equation}
and $\lambda_{i}=\lambda_{i}(x_0)\geq 0$. The induced Hermitian metric on 
$F:=L^{ k}\otimes E$ is given by  $|e_{F}(z)|_{h^F}^2=e^{-\psi(z)}$ with 
\begin{equation}\label{eq169}
\psi(z):=k\phi(z)+\varphi(z).
\end{equation}
The quadratic part of $\phi$ is denoted by 
\begin{equation}\label{eq170}
\phi_{0}(z):=\sum \lambda_{i}|z_{i}|^2.
\end{equation}

Assume $\alpha\in \Omega^{p,q}(X,F)$, then it has the form $\alpha=\xi\otimes e_{F}$  around $x_0\in K_1$ where 
$\xi=\sum f_{IJ}dz_I\wedge d\overline{z}_J$
 is a local $(p,q)$-form and $f_{IJ}$ are smooth functions on $W\subset \C^n$. The scaled 
 functions and sections with respect to $k\in \N$ are defined by 
 \begin{equation}\label{eq171}
 \psi^{(k)}(z):=\psi(z/\sqrt{k}),\quad e_{F}^{(k)}(z):=e_{F}(z/\sqrt{k})\,,
 \:\:\text{for $z\in\sqrt{k}W=B(\sqrt{k}c_0)$},
 \end{equation}
 hence $|e^{(k)}_{F}|_{h^F}^2=e^{-\psi^{(k)}}$. 
 The scaled forms are defined for $z\in\sqrt{k}W$ by 
\begin{equation}\label{eq149}
\begin{split}
&\omega^{(k)}(z):=\sqrt{-1}\sum h^{(k)}_{ij}(z) d{z_j} \wedge d{\overline{z_j}}
:=\sqrt{-1}\sum h_{ij}(z/\sqrt{k}) d{z_j} \wedge d{\overline{z_j}}, \\ 
&\xi^{(k)}(z):=f^{(k)}_{IJ}(z)dz_I\wedge d\overline{z_J}:=f_{IJ}(z/\sqrt{k})dz_I\wedge d\overline{z_J}, \\ 
&\alpha^{(k)}(z):=\xi^{(k)}(z)\otimes e_{F}^{(k)}(z)\,.
\end{split}
\end{equation} 
 
\begin{lemma}\label{lem3}
Let $(X,\omega)$ be a Hermitian manifold of dimension $n$ and
$(L,h^L)$ and $(E,h^E)$ be Hermitian holomorphic line bundles over $X$. 
Let $K_1$ and $K_2$ be compact subsets in $X$ such that $K_1\subset \mathring{K_2}$. 
Assume  $L \geq 0$ on $\mathring{K_2}$  and $q\geq 1$.  
Then there exists a constant $C>0$ depending on $\omega$, $K_1$, $K_2$, $(L, h^L)$ and $(E ,h_{E})$, such that 
\begin{equation}\label{eq161}
|\alpha(x_0)|_{h_k,\omega}^2\leq Ck^n \int_{|z|<\frac{2}{\sqrt{k}}}|\alpha|_{h_k,\omega}^2 dv_X.
\end{equation}
for any $x_0\in K_1$, $\alpha\in\cH^{n,q}(X,L^k\otimes E)$ and $k$ sufficiently large, 
where $|\cdot|_{h_{k},\omega}^2$ is the pointwise Hermitian norm induced by $\omega$, $h^L$ and $h^E$.
\end{lemma}   
\begin{proof}
Denote a ball centred at the origin in $\C^n$ with radius $r>0$ by $B(r):=\{z\in\C^n: |z|<r\}$. 
Then, $B(r)$ is a subset of $W\cong V\subset X$ via the local chart, when $0<r\leq c_0$. 
Thus the right side of (\ref{eq161}) is well defined, when $k$ is large enough. 
Let $r_k:=\frac{\log k}{\sqrt{k}}$ for $ k\in \N$. Then $0\leq r_k \leq 1$ and $r_k\rightarrow 0$ as $k\rightarrow \infty$. 

Under the local representation of forms valued in $F:=L^k\otimes E$, the Kodaira Laplacian 
can be represented by $\square=\ddbar\ddbar^{*}_{\psi}+\ddbar^{*}_{\psi}\ddbar$ 
locally over $B(\frac{\log k}{\sqrt{k}})$ for  $k$ large enough, where $\ddbar^{*}_{\psi}=\ddbar^{F*}$. 
Then the scaled Laplacian 
\begin{equation} \label{eq151}
\square^{(k)}:=\ddbar\ddbar^{*}_{\psi^{(k)}}+\ddbar^{*}_{\psi^{(k)}}\ddbar
\end{equation}
is well defined on $B({\log k})$ for $k$ large enough. 
 
Under our assumptions, we consider the harmonic $L^k\otimes E$-valued $(n,q)$-form 
$\alpha=\alpha\mid_{B(\frac{\log k}{\sqrt{k}})}$ on $B(\frac{\log k}{\sqrt{k}})$ 
for $k$ large enough, then the rescaled $\alpha^{(k)}$ is a $L^k\otimes E$-valued $(n,q)$-forms 
on $B({\log k})$ by (\ref{eq149}). By definitions,  
\begin{equation}\label{eq152}
\square^{(k)}\alpha^{(k)}=\frac{1}{k}(\square \alpha)^{(k)}=0
\end{equation}
over $B({\log k}).$ That is, the scaled forms are still harmonic with respect to the scaled Laplacian.
 
Next we introduce the following $L^2$-norms,
$$ \|\cdot\|^2_{B(\frac{2}{\sqrt{k}})}:=
\int_{B(\frac{2}{\sqrt{k}})}|\cdot|^2_{\omega} e^{-\psi}\omega_n \quad \mbox{and} \quad
  \|\cdot\|^2_{\phi_0,B(2)}:=\int_{B(2)}|\cdot|^2_{\beta}e^{-\phi_0}\beta_n.$$ 
  We claim that there exist $C(k)>0$ bounded above and below for $k$ large enough 
  (in fact, $C(k)\rightarrow 1$ as $k\rightarrow\infty$) such that
 \begin{equation}\label{eq158}
  \|\alpha^{(k)}\|^2_{\phi_0,B(2)}=C(k)k^n \|\alpha\|^2_{B({\frac{2}{\sqrt{k}}})}.
 \end{equation} 
 In fact, by (\ref{eq167})\textendash(\ref{eq171}),  $\psi^{(k)}(z)-\phi_0{(z)}=\frac{O(|z|^2)}{\sqrt{k}}$ and thus 
 \begin{equation}\label{eq150}
 \lim_{k\rightarrow \infty}\sup_{|z|<\log k}|\dbar^N (\psi^{(k)}-\phi_{0})(z)|=0,
 \end{equation}
 which means the scaled Hermitian metric on $L^k\otimes E$ convergences to a model metric on $B({\log k})$ with all derivatives. In particular, as $k\rightarrow\infty$,  $\psi^{(k)}(z)\rightarrow\phi_0(z)$ uniformly over $B({\log k})$, and also $\omega^{(k)}(z)\rightarrow \beta$. Then (\ref{eq158}) follows by
$$\|\alpha^{(k)}\|^2_{\phi_0,B(2)}=\int_{B(2)}|\xi^{(k)}(z)|^2_{\beta}e^{-\phi_0(z)}\beta_n~~,~
k^n \|\alpha\|^2_{B({\frac{2}{\sqrt{k}}})} 
=\int_{B(2)}|\xi^{(k)}(z)|^2_{\omega^{(k)}(z)}e^{-\psi^{(k)}(z)}\omega^{(k)}_n(z).$$
Finally, we apply \cite[Lemma 3.1]{RB:04} and identify $\alpha^{(k)}$ with a form in $L^2(\C^n,\phi_0)$ by extending with zero outside $B({\log k})$. Then there exists a constant $C_1>0$ independent of $k$ such that 
\begin{equation}\label{eq160}
\sup_{z\in B(1)}|\alpha^{(k)}(z)|^2_{\beta,\phi_0}\leq C_1\|\alpha^{(k)}\|^2_{\phi_0,B(2)}
\end{equation} 
for $k$ large enough, where $| \cdot|^2_{\beta,\phi_0}:=|\cdot|^2_{\beta}e^{-\phi_0}$.
Combining (\ref{eq158}) and (\ref{eq160}), we get $ |\alpha(x_0)|_{h_k,\omega}^2
  = |\alpha^{(k)}(0)|_{\beta,\phi_0}^2
  \leq 2C_1k^n \|\alpha\|^2_{B({\frac{2}{\sqrt{k}}})}$ for $k$ large enough. Notice that 
here $C_1$ works for all points sufficiently close to $x_0$ by continuity. 
That is, there exists a constant $C_1>0$ and a neighbourhood $B(x_0,\epsilon)$ of $x_0$, such that 
$ |\alpha(x)|_{h_k,\omega}^2\leq 2C_1 k^n \|\alpha\|^2_{B({\frac{2}{\sqrt{k}}})}$ for any 
$x\in B(x_0,\epsilon)$ and $k$ large enough. 
Since $K_1$ is compact, there exists a uniform constant $C>0$ which works for all $x\in K_1$, and (\ref{eq161}) follows. 
\end{proof}  
 
To summarize, we have a local estimate of the pointwise norm of harmonic forms valued in 
semipositive line bundles, which is equivalent to Theorem \ref{thm1}. We define 
\[
S_k^q(x):=\sup\left\{ \frac{|\alpha(x)|_{h_k,\omega}^2}{||\alpha||^{2}_{L^2}}: 
\alpha\in \cH^{n,q}(X,L^{ k}\otimes E)  \right\}, 
\]
where $(L,h^L)$ and $(E,h^E)$ are holomorphic Hermitian line bundles over a Hermitian manifold $(X, \omega)$ as before.
 
\begin{thm}\label{thm3}
Let $(X,\omega)$ be a Hermitian manifold of dimension $n$ and
$(L,h^L)$ and $(E,h^E)$ be Hermitian holomorphic line bundles over $X$. 
Let $K_1$ and $K_2$ be compact subsets in $X$ such that $K_1\subset \mathring{K_2}$. Assume  $L \geq 0$ on $\mathring{K_2}$  and $q\geq 1$.
Then there exists $C>0$ depending on $\omega$, $K_1$, $K_2$, $(L, h^L)$ and $(E ,h^{E})$ such that 
\begin{equation}\label{eq162}
S_k^q(x) \leq C k^{n-q}\,.
\end{equation}
for any $x\in K_1$ and $k\geq 1$.
\end{thm}
  
\begin{proof}
Combine (\ref{eq161}) and the case $r=\frac{2}{\sqrt{k}}$ of (\ref{eq4}).
\end{proof}
  
\begin{rem}
In particular, when $X$ is compact without boundary, and $K_1=K_2=X$, then (\ref{eq162}) implies the case of $\lambda=0$ in \cite[Theorem 2.3]{BB:02}.
\end{rem}

               \section{Proof of the main results and a corollary}\label{proof}

At first, the following lemma is clear by the definitions of $S^q_k(x)$ and $B^q_k(x)$ in Theorem \ref{thm3} and Theorem \ref{thm1}. 
\begin{lemma}\label{lem4}
$S^q_k(x)\leq B^q_k(x)\leq C_{n}^q S^q_k(x)$ on $X$.
\end{lemma}
By this lemma and the submeanvalue formulas of harmonic forms in $\cH^{n,q}(X,L^k\otimes E)$ in the section \ref{submeanvalue}, we can prove the first main result immediately.

\textbf{Proof of  Theorem \ref{thm1}}: Assume $U$ is a neighbourhood of $K$ such that $L$ is semipositive on $U$. Then we choose $K_2$ such that $K_1\subset \mathring{K_2} \subset K_2 \subset U$, and apply Theorem \ref{thm3} and Lemma \ref{lem4}.  \qed  

Now we can prove the second main results on $\Gamma$-dimension and covering manifolds.

\textbf{Proof of  Theorem \ref{thm2}}:
Under the assumption of $X$, $\Gamma$ and $X/\Gamma$, there exists an open fundamental domain $U\subset X$ of the action $\Gamma$ on $X$ such that the closure $\overline{U}$ is compact.  

Since $L$ and $E$ are $\Gamma$-invariant holomorphic Hermitian line bundles over $X$, the induced Hermitian line bundle $F:=L^k\otimes E$ is also $\Gamma$-invariant and holomorphic. Then the Kodaira Laplacian $\square:=\square^{F}$ is $\Gamma$-invariant. Thus $\square$ is essentially self-adjoint(cf. \cite{MM} Corollary 3.3.4), and we denote still by $\square$ its self-adjoint extension, which commutes to the action of $\Gamma$. According to (\cite[Lemma C.3.1, Lemma 3.6.3]{MM}), the space of harmonic $F$-valued $(n,q)$-forms 
$\cH^{n,q}(X,F)$ is a $\Gamma$-module on which $\Gamma$-dimension is well-defined. By \cite[(3.6.11), (3.6.17)]{MM}, we have 
\begin{equation}\label{eq56}
\dim_{\Gamma}\cH^{n,q}(X,F)
=\sum_{i}\int_{U}|s_i(x)|_{h^F,\omega}^2 d v_{X}(x),
\end{equation}
where $\{s_i\}$ is an orthonormal basis of $\cH^{n,q}(X,F)$ with respect to the scalar product in $L^2_{n,q}(X,F)$. Using Theorem \ref{thm1}, we have
\begin{equation}\label{eq168}
B^q_k(x):=\sum_{i}|s_i(x)|_{h_k,\omega}^2\leq C k^{n-q}
\end{equation}
for any $x\in U$. Then integrating $B^q_k(x)$ over $U$ and combining (\ref{eq168}) and (\ref{eq56}), we obtain
$\dim_{\Gamma}\cH^{n,q}(X,L^k\otimes E)
\leq C k^{n-q}$, that is the first asymptotic estimate in (\ref{eq2}). 
   
 Let $\overline{H}^{0,q}_{(2)}(X,L^k\otimes E)$ be the reduced $L^2$-Dolbeault cohomology group, which is canonically isomorphic to $\cH^{0,q}(X,L^k\otimes E)$ as $\Gamma$-modules by the weak Hodge decomposition, thus $\dim_{\Gamma}{\overline{H}}^{0,q}_{(2)}(X,L^k\otimes E)
=\dim_{\Gamma}\cH^{0,q}(X,L^k\otimes E)$. Substituting $E \bigotimes \Lambda^n (T^{(1,0)}X)$ for $E$ in the first estimate of (\ref{eq2}), then the same asymptotic estimate also holds for the space of harmonic $L^{ k}\otimes E$ valued $(0,q)$-forms and (\ref{eq16}) follows.\qed  
     
\begin{cor}\label{cor2} 
Let $(X,\omega)$ be a compact Hermitian manifold  of dimension $n$ and $(E, h^E)$ be a semipositive holomorphic Hermitian vector bundle of rank $r$ (i.e. $ L(E^{*})^{*}\geq 0$). Then there exists $C>0$ such that for any $q\geq 1$ and $k\geq 1$ we have
\begin{equation}\label{eq15}
\dim H^{q} (X, S^k (E))\leqq C k^{(n+r-1)-q}
\end{equation} 
where $S^k (E)$ is the $k$-th symmetric tensor power of $E$.
\end{cor}
\begin{proof}
We assume $\Gamma$ and $E$ are trivial in (\ref{eq16}), and notice the theorem of Le Potier (cf. \cite[Chap.III \S 5 (5.7)]{Kob:87} ), which relates vector bundle cohomology to line bundle chohomology, then
$\dim H^{q} (X, S^k (E))= \dim H^{q} (P(E^*), (L(E^{*})^{*})^k) \leqq C k^{(n+r-1)-q}.$
\end{proof}
  
\begin{rem}
Let $(X,\omega)$ be a complete Hermitian manifold of dimension $n$ and
$(L,h^L)$ and $(E,h^E)$ be Hermitian holomorphic line bundles over $X$. 
Suppose $X,L$ and $E$ have bounded geometry and  $L\geq 0$ over $X$. By Theorem \ref{thm1}, there exists a constant $C>0$ such that the Bergman kernel function $B^q_k(x)\leq C k^{n-q}$ for any $x\in X$, $k\geq 1$ and $q\geq 1$.
\end{rem} 

\begin{rem}
Let $(X,\omega)$ be a Hermitian manifold of dimension $n$ on which a discrete 
group $\Gamma$ acts holomorphically, freely and properly such that $\omega$ is a $\Gamma$-invariant Hermitian 
metric and the quotient $X/\Gamma$ is compact. Let $(L,h^L)$ be a $\Gamma$-invariant 
holomorphic Hermitian line bundle on $X$. Assume $L\leq 0$ is semi-negative (i.e. $L^*\geq 0$). According to Serre duality (cf. \cite[3.15]{CD:01}) and Theorem \ref{thm2},
there exists $C>0$ such that for any 
$q\leq n-1$ and  $k\geq 1$ we have
$$\dim_{\Gamma}{\overline{H}}^{0,q}_{(2)}(X, L^k)=
\dim_{\Gamma}{\overline{H}}^{n,n-q}_{(2)}(X, L^{*k})=\dim_{\Gamma}{\overline{H}}^{0,n-q}_{(2)}(X, \Lambda^n (T^{*(1,0)}X)\otimes L^{*k})\leq Ck^q.$$ In particular, for all $k\in \N$, $\dim_{\Gamma}{\overline{H}}^{0,0}_{(2)}(X, L^k)
\leq C$. 
\end{rem} 

\textbf{Acknowledgements}. The author would like to thank Professor George Marinescu for his kind advice. This work was supported by China Scholarship Council.

\end{document}